\documentclass[reqno]{amsart}

\usepackage{latexsym}
\usepackage{amsmath}
\usepackage{amssymb}
\usepackage{amscd}
\usepackage{multicol}
\usepackage{amsfonts}
\usepackage{amsthm}
\usepackage{color}
\usepackage{graphicx}
\usepackage{epsfig}
\usepackage{psfrag}

\input xy
\xyoption{all}

\theoremstyle{definition}
\newtheorem{thm}{Theorem}
\newtheorem{theorem}[thm]{Theorem}
\newtheorem{corollary}[thm]{Corollary}
\newtheorem{lemma}[thm]{Lemma}

\newtheorem{remark}[thm]{Remark}

\newtheorem{exm}{Example}
\newtheorem*{exm*}{Example}
\numberwithin{equation}{section}
\numberwithin{thm}{section}

\newcommand{\ft}{{\mathfrak{t}}}

\newcommand{\fg}{{\mathfrak{g}}}
\newcommand{\fk}{{\mathfrak{k}}}

\newcommand{\g}{\mathfrak{t}}
\newcommand{\gd}{\mathfrak{t}^{*}}

\renewcommand{\SS}{{\mathbb{S}}}
\newcommand{\RR}{{\mathbb{R}}}
\newcommand{\R}{\mathbb{R}}
\newcommand{\ZZ}{{\mathbb{Z}}}
\newcommand{\CC}{{\mathbb{C}}}

\newcommand{\cR}{{\mathcal{R}}}

\title{Balanced Fiber Bundles and GKM Theory}
\author[V. Guillemin]{Victor Guillemin}
\address{Department of Mathematics, MIT, Cambridge, MA 02139}
\email{vwg@math.mit.edu}

\author[S. Sabatini]{Silvia Sabatini}
\address{Department of Mathematics, EPFL, Lausanne, Switzerland}
\email{silvia.sabatini@epfl.ch}

\author[C. Zara]{Catalin Zara}
\address{Department of Mathematics, University of Massachusetts
Boston, MA 02125}
\email{czara@math.umb.edu}

\date{May 2, 2011}

\begin{document}

\begin{abstract}
Let $T$ be a torus and $B$ a compact $T-$manifold.
Goresky, Kottwitz, and MacPherson show in \cite{GKM}
that if $B$ is (what was subsequently called) a GKM
manifold, then there exists a simple combinatorial
description of the equivariant cohomology ring $H_T^*(B)$
as a subring of $H_T^*(B^T)$. In this paper
we prove an analogue of this result for $T-$equivariant
fiber bundles: we show that if $M$ is a $T-$manifold and
$\pi \colon M \to B$ a fiber bundle for which $\pi$
intertwines the two $T-$actions, there is a simple
combinatorial description of $H_T^*(M)$ as a subring of
$H_T^*(\pi^{-1}(B^T))$. Using this result we obtain fiber 
bundle analogues of results of \cite{GHZ} on GKM theory 
for homogeneous spaces.
\end{abstract}

\maketitle

\tableofcontents

\section{Introduction}

Let $T=(S^1)^n$ be an $n-$dimensional torus and $M$ a compact,
connected $T-$manifold.
We recall that the equivariant cohomology
$H_T^*(M) = H_T^*(M; \RR)$ of $M$ is defined as the
usual cohomology of the quotient $(M \times E)/T$,
where $E$ is the total space of the classifying bundle
of the group $T$. Let
\begin{equation}\label{eq:1.1}
  \pi \colon M \to B
\end{equation}
be a $T-$equivariant fiber bundle. We will
assume that the base $B$ is simply
connected and that the typical fiber is connected.

Then one gets a fiber bundle
\begin{equation}\label{eq:1.2}
  (M \times E)/T \to (B \times E)/T
\end{equation}
and a Serre-Leray spectral sequence relating the
equivariant cohomology groups of $M$ and $B$;
the $E_2-$term of this spectral sequence is the product
\begin{equation}\label{eq:1.3}
  H^*(F) \otimes H^*((B\times E)/T)
\end{equation}
where $F$ is the fiber of the bundle \eqref{eq:1.2}
and hence of the bundle \eqref{eq:1.1}. Thus if the
spectral sequence collapses at this stage, one
gets an isomorphism of
additive cohomology
\begin{equation}\label{eq:1.4}
  H_T^*(M) \simeq H^*(F) \otimes H_T^*(B) \; .
\end{equation}
However, this isomorphism doesn't say much about how
the ring structure of $H_T^*(B)$ and $H_T^*(M)$ are related.
One of the main goals of this paper is to address that
question. We begin by recalling that one approach for
computing the equivariant cohomology ring of a $T-$manifold
$M$ is by Kirwan localization. Namely, if $H_T^*(M)$ is
torsion-free, the restriction map
\begin{equation*}
 i^* \colon H_T^*(M) \to H_T^*(M^T)
\end{equation*}
is injective and hence computing $H_T^*(M)$ reduces to
computing the image of $H_T^*(M)$ in $H_T^*(M^T)$.
If $M^T$ is finite, then
$$H_T^*(M^T) = \bigoplus_{p\in M^T} \SS(\ft^*)\; ,$$
with one copy of $H_T^*(pt) \simeq \SS(\ft^*)$
for each $p \in M^T$, where
$\SS(\ft^*)$ is the symmetric algebra of $\ft^*$.
Determining where $H_T^*(M)$ sits
inside this sum is a challenging problem in combinatorics.
However, one class of spaces for which this problem has a
simple and elegant solution is the one introduced by
Goresky-Kottwitz-MacPherson in their seminal paper \cite{GKM}.
These are now known as \emph{GKM spaces}, a $T-$manifold $M$ being ``GKM" if
\begin{itemize}
\item[(a)] $M^T$ is finite
\item[(b)] $M$ is equivariantly formal, i.e.
$$
H_T(M)\simeq H(M)\otimes_{\mathbb{C}}\SS(\ft^*)
$$
as $\SS(\ft^*)$ modules.
\item[(c)]
 For every codimension
one subtorus $T' \subset T$, the connected components
of $M^{T'}$ are either points or two-spheres. 
\end{itemize}
If $S$ is one of the edge two-spheres, then $S^T$ consists
of exactly two $T-$fixed points, $p$ and $q$ (the ``North''
and ``South'' poles of $S$). To each GKM space $M$ we attach
a graph $\Gamma=\Gamma_M$ with set of vertices $V_\Gamma=M^T$,
and edges corresponding to these two-spheres.
If $M$ has an invariant almost complex or symplectic structure,
then the isotropy representations on tangent spaces at fixed points
are complex representations and their weights are well-defined.

These data determine a map
$$\alpha \colon E_{\Gamma} \to \ZZ_T^*$$
of oriented edges of $\Gamma$ into the weight lattice of $T$.
This map assigns to the edge (two-sphere) $S$, joining $p$
to $q$ and oriented from $p$ to $q$, the weight of the isotropy
representation of $T$ on the tangent space to $S$ at $p$.
The map $\alpha$ is called the \emph{axial function} of the
graph $\Gamma$. We use it to define a subring
$H_{\alpha}^*(\Gamma)$ of $H_T^*(M^T)$ as follows.
Let $c$ be an element of $H_T^*(M^T)$, \emph{i.e.} a function
which assigns to each $p\in M^T$ an element $c(p)$ of
$H_T^*(pt)=\mathbb{S}(\mathfrak{t}^*)$. Then $c$
is in $H_{\alpha}^*(\Gamma)$ if and only if for each edge
$e$ of $\Gamma$ with vertices $p$ and $q$ as endpoints,
$c(p)\in \mathbb{S}(\mathfrak{t}^*)$ and
$c(q)\in \mathbb{S}(\mathfrak{t}^*)$ have the same image
in $\mathbb{S}(\mathfrak{t}^*)/\alpha_e\mathbb{S}(\mathfrak{t}^*)$.
(Without the invariant almost complex or symplectic structure, the
isotropy representations are only real representations and the weights
are defined only up to sign; however, that does not change the
construction of $H_{\alpha}^*(\Gamma)$.)

For GKM spaces, a direct consequence of a theorem of Chang and Skjelbred
(\cite{CS}) is that $H_{\alpha}^*(\Gamma)$
is the image of $i^*$ (see \cite{GKM}), and
therefore there is an isomorphism of rings
\begin{equation}\label{eq:1.6}
 H_T^*(M) \simeq H_{\alpha}^*(\Gamma) \; .
\end{equation}
One of our main results is a generalization
of \eqref{eq:1.6} for $T-$equivariant fiber bundles
\begin{equation}\label{eq:fiber_bundle}
\pi \colon M \to B
\end{equation}
for which the total space $M$ is
equivariantly formal and the base $B$ is a GKM space.
By the Kirwan Theorem the composite map
$$H_T^*(M) \to H_T^*(\pi^{-1}(B^T)) \to H_T^*(M^T)$$
is injective. Hence one has an injective homomorphism of rings
\begin{equation}\label{eq:new1.6}
H_T^*(M) \to \bigoplus_{p \in B^T} H_T^*(\pi^{-1}(p))\; ,
\end{equation}
and so to determine the ring structure of $H_T^*(M)$ it
suffices to determine the image of this mapping. This we will
do by a GKM type recipe similar to \eqref{eq:1.6}.

Let $(\Gamma=\Gamma_B, \alpha)$
be the GKM graph associated to $B$, and
for $p \in B^T$ (\emph{i.e.} a vertex of $\Gamma$) let
$F_p = \pi^{-1}(p)$. If $e$
is an edge of $\Gamma$ joining the vertices $p$ and $q$,
and $T_e$ is the subtorus of $T$ with Lie algebra $\ker \alpha_e$, 
then $F_p$ and $F_q$ are isomorphic as $T_e-$spaces and hence
\begin{equation}\label{eq:new1.7}
H_T^*(F_p)/\langle \alpha_e \rangle = H_T^*(F_q)/\langle \alpha_e \rangle \; ,
\end{equation}
and denoting the ring \eqref{eq:new1.7} by $\cR_e$,
we will prove the following generalization of~\eqref{eq:1.6}.

\begin{thm}\label{th:CS-bundles}
  A function
  $$c \colon V_\Gamma \to \bigoplus_{p \in B^T} H_T^*(F_p), \; c(p) \in H_T^*(F_p)$$
  is in the image of \eqref{eq:new1.6} if and only if for every edge
  $e=(p,q)$ of $\Gamma$, the images of $c(p)$ and $c(q)$
  in $\cR_e$ coincide.
\end{thm}

One of our main applications of this result will be a fiber 
bundle version of the main result in \cite{GHZ}. In more detail:
In \cite{GHZ} it is shown that if $G$ is a compact semisimple Lie group, 
$T$ a Cartan subgroup, and $K$ a closed subgroup of $G$, then the 
following conditions are equivalent:
\begin{enumerate}
\item The action of $T$ on $G/K$ is GKM;
\item The Euler characteristic of $G/K$ is non-zero;
\item $K$ is of maximal rank, \textit{i.e.} $T \subset K$.
\end{enumerate}

Moreover, for homogeneous spaces of the form $G/K$ one has a description, 
due to Borel, of the equivariant cohomology ring of $G/K$ as a tensor product
\begin{equation}\label{eq:BorelGK}
H_T^*(G/K) = \SS(\ft^*)^{W_K} \oplus_{\SS(\ft^*)^{W_G}} \SS(\ft^*)
\end{equation}
and it is shown in \cite{GHZ} how to reconcile this description 
with the description \eqref{eq:1.6}.

Our fiber bundle version of this result will be a a description of the 
cohomology ring of $G/K_1$, with $K_1 \subset K$, in terms of the fiber bundle 
$G/K_1 \to G/K$, a description that will be of Borel type on the fibers and 
of GKM type on the base. This result will (as we've already shown in 
special cases in \cite{GSZ}) enable one to interpolate between two 
(in principle) very different descriptions of the ring $H_T(G/K)$.

The fibrations $G/K_1 \to G/K$ are special cases of a class of fibrations 
which come up in many other context as well (for instance in the 
theory of toric varieties) and which for the lack of a better name 
we will call \emph{balanced fibrations}.

To explain what we mean by this term let $F_p$ and $F_q$ be as in 
\eqref{eq:new1.7}. Then there is a diffeomorphism $f_e \colon F_p \to F_q$, 
canonical up to isotopy, which is $T_e-$invariant but in general not 
$T-$invariant. We will say that the fibration $M \to B$ is balanced at 
$e$ if one can twist the $T$ action on $F_q$ to make $f_e$ be $T-$invariant, 
\textit{i.e.} if one can find an automorphism $\tau_e \colon T \to T$, 
restricting to the identity on $T_e$, such that 
\begin{equation}
f_e (gx) = \tau_e(g)f_e(x)
\end{equation}
for all $g \in T$ and $x \in F_p$. (Since $f_e$ is unique up to isotopy, 
this $\tau_e$, if it exists, is unique.)

Suppose now that the $T$ action on $M$ is balanced in the sense that it 
is balanced at all edges $e$.
Then, denoting by $Aut(F_p)$ the group of isotopy classes of diffeomorphisms of $F_p$
and by $Aut(T)$ the group of automorphisms of $T$, one gets a homomorphism of the loop group 
$\pi_1(\Gamma,p)$ into $Aut(F_p) \times Aut(T)$ mapping the loop of edges, $e_1,\ldots,e_k$,
to $(f_{e_k}\circ \cdots \circ f_{e_1},\tau_{e_k}\circ \cdots \circ \tau_{e_1} )$.
The image of this map we'll denote by $W_p$ and call the \textit{Weyl group of $p$}. This
group acts on $H_T(F_p)$ and, modulo some hypotheses on $M$ which we'll spell out more carefully in section 5, we'll show that there
is a canonical imbedding of $H_T(F_p)^{W_p}$ into $H_T(M)$ and that its image generates
$H_T(M)$ as a module over $H_T(B)$. More explicitily, we will 
show that 
\begin{equation}\label{eq:newmain}
H_T(M) = H_T(F_p)^{W_p} \otimes_{\SS(\ft^*)^{W_p}} H_{\alpha}(\Gamma_B)
\end{equation}
where $\Gamma_B$ is the GKM graph of $B$.

A few words about the organization of this paper. In Section\ref{sec:CS} 
we will generalize the Chang-Skjelbred theorem to equivariant fiber bundles, 
and in Section~\ref{sec:GKMFB} use this result to prove 
Theorem~\ref{th:CS-bundles}. 
In Section~\ref{sec:4} we will describe 
in more detail the results of \cite{GHZ} alluded 
to above and show that the fibrations $G/K_1 \to G/K$ are balanced.
Then in Section~\ref{sec:5} we will verify \eqref{eq:newmain} and in 
Section~\ref{sec:6} 
describe some connections between the results of this paper and results 
of \cite{GSZ} (where we work out the implications of this theory in 
much greater detail for the classical flag varieties of type 
$A_n$, $B_n$, $C_n$, and $D_n$).

The results of this paper are also related to the results of 
\cite{GZ2}, the topic of which is
K-theoretic aspects of GKM theory. (In some work-in-progress we are
investigating the implications of these results for GKM fibrations. In
particular we are able to show that there is a K-theoretic analogue of the
Chang-Skjelbred theorem of Section~\ref{sec:CS} and that it gives one an effective way
of computing the equivariant K-groups of balanced fiber bundles.)

\section{The Chang-Skjelbred Theorem for Fiber Bundles}
\label{sec:CS}

Let $\pi \colon M \to B$ be a $T-$equivariant fiber bundle
with $M$ equivariantly formal
and $B$ a GKM space. Let $K_i,\;i=1,\ldots,N$ be the codimension one isotropy groups
of $B$ and let $\mathfrak{k_i}$ be the Lie algebra of $K_i$. Since $B$ is GKM one has
the following result.
\begin{lemma}
If $K$ is an isotropy group of $B$, and $\fk$ is the Lie algebra of $K$, then
$$
\fk =\bigcap_{r=1}^m \mathfrak{k}_{i_r}
$$
for some multi-index $1\leqslant i_1< \ldots < i_m \leqslant N$.
\end{lemma}

For $K$ a subgroup of $T$ let $X^K=\pi^{-1}(B^K)$, where $B^K\subset B$ denotes the set
of points in $B$ fixed by $K$.
We recall (\cite[Section 11.3]{GS}) that if $A$ is a finitely generated $\SS(\gd)-$module,
then the annihilator ideal of $A$, $I_A$ is defined to be
$$
I_A=\{f\in \SS(\gd),\;fA=0\}\;,
$$
and the support of $A$ is the algebraic variety in $\g\otimes \CC$ associated with this ideal,
\textit{i.e.}
$$
\mbox{supp}A=\{x\in \g \otimes \CC, f(x)=0 \mbox{ for all }f\in I_A\}\;.
$$

Then from the lemma and \cite[Theorem 11.4.1]{GS}  one gets the following.

\begin{theorem}\label{support}
The $\SS(\gd)-$module $H_T^*(M\setminus X^T)$ is supported on the set
\begin{equation}\label{eq:2.1}
\displaystyle\bigcup_{i=1}^N \mathfrak{k}_i\otimes \CC
\end{equation}
\end{theorem}

By \cite[Section 11.3]{GS} there is an exact sequence
\begin{equation}\label{eq:2.2}
H_T^k(M \setminus X^T)_c \longrightarrow H_T^k(M)\stackrel{i^*}{\longrightarrow}
H_T^k(X^T)\longrightarrow H_T^{k+1}(M\setminus X^T)_c
\end{equation}
where $H_T^*(\;\cdot\;)_c$ denotes the equivariant cohomology with compact supports.
Therefore since $H_T^*(M)$ is a free $\SS(\gd)-$module Theorem \ref{support} implies
the following theorem.
\begin{theorem}\label{ker and coker}
The map $i^*$ is injective and $\mbox{coker}(i^*)$ is supported on
$\displaystyle\bigcup_{i=1}^N \mathfrak{k}_i\otimes \CC$.
\end{theorem}
As a consequence we get the following corollary.
\begin{corollary}\label{cor:2.4}
If $e$ is an element of $H_T^*(X^T)$, there exist non-zero weights
$\alpha_1,\ldots,\alpha_r$ such that $\alpha_i=0$ on some $\mathfrak{k}_j$ and
\begin{equation}\label{eq:2.3}
\alpha_1\cdots\alpha_r e \in i^*(H_T^*(M))
\end{equation}
\end{corollary}

The next theorem is a fiber bundle version of the Chang-Skjelbred theorem.

\begin{theorem}\label{thm:CSt}
The image of $i^*$ is the ring
\begin{equation}\label{eq:2.4}
\bigcap_{i=1}^N i^*_{K_i}H_T^*(X^{K_i})
\end{equation}
where $i_{K_i}$ denotes the inclusion of $X^T$ into $X^{K_i}$.
\end{theorem}

\begin{proof}
Via the inclusion $i^*$ we can view $H_T^*(M)$ as a submodule of $H_T^*(X^T)$.
Let $e_1,\ldots,e_m$ be a basis of $H_T^*(M)$ as a free module
over $\SS(\gd)$. Then by Corollary~\ref{cor:2.4}  for any $e\in H_T^*(X^T)$
we have $$
\alpha_1\cdots\alpha_r e=\sum f_ie_i,\;f_i\in \SS(\gd)\;.
$$

Then $e=\sum\displaystyle\frac{f_i}{p}e_i$, where $p=\alpha_1\cdots\alpha_r$.
If $f_i$ and $p$ have a common factor we can eliminate it and write $e$ uniquely as
\begin{equation}\label{eq:2.5}
e=\sum \frac{g_i}{p_i}e_i
\end{equation}
with $g_i\in \SS(\gd)$, $p_i$ a product of a subset of the weights
$\alpha_1,\ldots,\alpha_r$ and $p_i$ and $g_i$ relatively prime.

Now suppose that $K$ is an isotropy subgroup of $B$ of codimension one and $e$
is in the image of $H^*_T(X^K)$. By \cite[Theorem 11.4.2]{GS} the cokernel
of the map $H^*_T(M)\rightarrow H_T^*(X^K)$ is supported on the subset
$\cup \mathfrak{k}_i\otimes \CC,\;\mathfrak{k}_i\neq \mathfrak{k}$ of \eqref{eq:2.1},
and hence there exists weights $\beta_1,\ldots,\beta_r$, $\beta_i$ vanishing
on some $\mathfrak{k}_j$ but not on $\mathfrak{k}$, such that
$$
\beta_i\cdots\beta_s e=\sum f_ie_i\;.
$$
Thus the $p_i$ in \eqref{eq:2.5}, which is a product of a subset of the
weights $\alpha_1,\ldots,\alpha_r$, is a product of a subset of weights
none of which vanish on $\mathfrak{k}$. Repeating this argument for all
the codimension one isotropy groups of $B$ we conclude that the weights in
this subset cannot vanish on any of these $\mathfrak{k}$'s, and hence is the
empty set, i.e. $p_i=1$. Then if $e$ is in the intersection \eqref{eq:2.4},
$e$ is in $H_T^*(M)$.
\end{proof}

\section{Fiber Bundles over GKM Spaces}
\label{sec:GKMFB}

Suppose now that $B=\CC P^1$. The action of $T$ on $B$ is effectively an
action of a quotient group, $T/T_e$, where $T_e$ is a codimension one
subgroup of $T$. Moreover $B^T$ consists of two points, $p_i,\;i=1,2,$
and $X^T$ consists of the two fibers $\pi^{-1}(p_i)=F_i$.
Let $T=T_e\times S^1$. Then $S^1$ acts freely on $\CC P^1\setminus \{p_1,p_2\}$
and the quotient by $S^1$ of this action is the interval $(0,1)$, so one
has an isomorphism of $T_e$ spaces
\begin{equation}\label{eq:2.6}
(M\setminus X^T)/S^1=F\times(0,1)\;,
\end{equation}
where, as $T_e-$spaces, $F=F_1=F_2$.

Consider now the long exact sequence \eqref{eq:2.2}. Since $i^*$ is injective
this becomes a short exact sequence
\begin{equation}\label{eq:2.7}
0\rightarrow H_T^k(M)\rightarrow H_T^k(X^T)\rightarrow H_T^{k+1}(M\setminus X^T)_c\rightarrow0\;.
\end{equation}
Since $S^1$ acts freely on $M\setminus X^T$ we have $$
H_T^{k+1}(M\setminus X^T)_c=H_{T_1}^{k+1}((M\setminus X^T)/S^1)_c
$$
and by fiber integration one gets from \eqref{eq:2.6}
$$
H_{T_e}^{k+1}((M\setminus X^T)/S^1)_c=H_{T_e}^k(F)\;.
$$
Therefore, denoting by $r$ the forgetful map $H_T(F_i) \to H_{T_e}(F)$,
the sequence \eqref{eq:2.7} becomes
\begin{equation}\label{eq:2.8}
0\rightarrow H_T^k(M) \rightarrow H_T^k(F_1)\oplus H_T^k(F_2) \rightarrow H_{T_e}^k(F)\rightarrow 0\; ,
\end{equation}
where the second arrow is the map
$$H_T(X^T) = H_T(F_1) \oplus H_T(F_2) \to H_{T_e}(F)$$
sending $f_1\oplus f_2$ to $-r(f_1) + r(f_2)$. (The $-r$ in the first term
is due to the fact that the fiber integral
$$H_c^{k+1} (F \times (0,1)) \to H^k(F)$$
depends on the orientation of $(0,1)$: the standard orientation for
$F_2 \times (0,1) \to F_2$ and the reverse orientation for $F_1 \times (0,1) \to F_1$.) To summarize, we've proved the following theorem.

\begin{thm}\label{thm:7.6}
  For $T-$equivariant fiber bundles over $\CC P^1$, the image of the map
  $$0 \to H_T^*(M) \to H_T^*(X^T)$$
  is the set of pairs $(f_1,f_2) \in H_T^*(F_1) \oplus H_T^*(F_2)$ satisfying $r(f_1) = r(f_2)$.
\end{thm}

Theorem~\ref{th:CS-bundles} follows from Theorem~\ref{thm:CSt} by applying
Theorem~\ref{thm:7.6} to all edges of the GKM graph of $B$.

\section{Homogeneous Fibrations}\label{sec:4}

Let $G$ be a compact connected semisimple Lie group, $T$ its Cartan subgroup, 
and $K$ a closed subgroup of $G$ containing $T$. Then, as asserted above, $G/K$ 
is a GKM space. The proof of this consists essentially of describing explicitly 
the GKM structure of $G/K$ in terms of the Weyl groups of $G$ and $K$. 
We first note that for $K=T$, i.e.  for the generalized flag variety $M=G/T$, 
the fixed point set, $M^T$, is just the orbit of $N(T)$ through the identity 
coset, $p_0$, and hence $M^T$ can be identified with $N(T)/T = W_G$. To show 
that $M$ is $GKM$ it suffices to check the GKM condition $p_0$. To do so we 
identify the tangent space $T_{p_0}M$ with $\fg/\ft$ and identify $\fg/\ft$ 
with the sum of the positive root spaces
\begin{equation}
\fg/\ft = \bigoplus_{\alpha \in \Delta_G^+} \fg_{\alpha}, \; ,
\end{equation}
the $\alpha'$s being the weights of the isotropy representation of $T$ 
on $\fg/\ft$. It then follows from a standard theorem in Lie theory 
that the weights are pairwise 
independent and this in turn implies ``GKM-ness'' at $p_0$.

To see what the edges of the GKM graph are at $p_0$ let 
$\chi_{\alpha} \colon T \to S^1$ be the character homomorphism
$$\chi_{\alpha}(t) = \exp{ i\alpha(t)} \; ,$$
let $H_{\alpha}$ be its kernel, and $G_{\alpha}$ the semisimple 
component of the centralizer $C(H_{\alpha})$ of $H_{\alpha}$ in $G$. Then 
$G_{\alpha}$ is either $SU(2)$ or $SO(3)$, and in either case 
$G_{\alpha} p_0 \simeq \CC P^1$.  However since $G_{\alpha}$ centralizes $H_{\alpha}$,
$G_{\alpha}p_0$ is $H_{\alpha}-$fixed and hence is the connected component of $M^{H_{\alpha}}$ containing $p_0$. Thus the oriented edges of the GKM graph of $M$ with initial point $p_0$
can be identified with the elements of $\Delta_G^+$ and the axial function becomes the function which labels by $\alpha$ the oriented edge $G_{\alpha}p_0$.
Moreover, under the identification $M^T=W_G$, the vertices that are joined to $p_0$
by these edges are of the form $\sigma_{\alpha}p_0$, where $\sigma_{\alpha}\in W_G$
is the reflection which leaves fixed the hyperplane $\ker \alpha \subseteq \ft$
and maps $\alpha$ to $-\alpha$.

Letting $a\in N(T)$ and letting $p=ap_0$ one gets essentially the same description of the
GKM graph at $p$. Namely, denoting this graph by $\Gamma$, the following are true.
\begin{enumerate}
\item The maps, $a\in N(T)\rightarrow ap_0$ and $a\in N(T)\to w\in N(T)/T$, set up
a one-one correspondence between the vertices, $M^T$, of $\Gamma$ and the elements of
$W_G$;
\item Two vertices, $w$ and $w'$, are on a common edge if and only if $w'=w\sigma_{\alpha}$
for some $\alpha\in \Delta_G^+$;
\item The edges of $\Gamma$ containing $p=ap_0$ are in one-one correspondence with elements of $\Delta_G^+$;
\item For $\alpha\in \Delta_G^+$ the stabilizer group of the edge corresponding to $\alpha$
is $aH_{\alpha}a^{-1}$.
\end{enumerate}
Via the fibration $G/T \to G/K$ one gets essentially the same picture for $G/K$.
Namely let $\Delta_{G,K}^+=\Delta_G^+\setminus \Delta_K^+$. Then one has (see \cite{GHZ}, Theorem 2.4)
\begin{theorem}
$G/K$ is a GKM space with GKM graph $\Gamma$, where
\begin{enumerate}
\item The vertices of $\Gamma$ are in one-one correspondence with the elements of
$W_G/W_K$;
\item Two vertices $[w]$ and $[w']$ are on a common edge if and only if $[w']=[w\sigma_{\alpha}]$
for some $\alpha \in \Delta_{G,K}^+$;
\item The edges of $\Gamma$ containing the vertex $[w]$ are in one-one correspondence with the roots in $\Delta_{G,K}^+$;
\item If $\alpha$ is such a root the the stabilizer group of the $\CC P^1$ corresponding to the edge
is $aH_{\alpha}a^{-1}$ where $a$ is a preimage in $N(T)$ of $w\in W_G$.
\end{enumerate}
\end{theorem}
\begin{remark}
The GKM graph that we have just described is not simple in general, i.e. will in general have more than one edge joining two adjacent vertices. There is, however, a simple sufficient condition for simplicity.
\end{remark}
\begin{theorem}
If $K$ is a stabilizer group of an element of $\ft^*$, i.e. if $G/K$ is a coadjoint orbit, then the
graph we've constructed above is simple.
\end{theorem}
Now let $K_1$ be a closed subgroup of $K$ and consider the fibration 
\begin{equation}\label{eq:fibration}
G/K_1\to G/K.
\end{equation}
To show that this is balanced it suffices to show that it is balanced at the edges going out
of the identity coset, $p_0$. However, if $e$ is the edge joining $p_0$ to $\sigma_{\alpha}p_0$
and $a$ is the preimage of $\sigma_{\alpha}$ in $N(T)$ then conjugation by $a$ maps
the fiber, $F_{p_0}=K/K_0$ of \eqref{eq:fibration} at $p_0$ onto the fiber $F_p:=aK/aK_0$ of
$\eqref{eq:fibration}$ at $p=\sigma_{\alpha}p_0$, and conjugates the action of $\Gamma$
on $F_{p_0}$ to the twisted action, $aTa^{-1}$, of $T$ on $F_p$.
Moreover, since $a$ is in the centralizer of $H_{\alpha}$, this twisted action, restricted to
$H_{\alpha}$, coincides with the given action of $H_{\alpha}$, i.e. if $T_e=H_{\alpha}$, conjugation
by $a$ is a $T_e-$equivariant isomorphism of $F_{p_0} $ onto $F_p$. Hence the fibration 
\eqref{eq:fibration} is balanced.

\section{Holonomy for Balanced Bundles}\label{sec:5}
Let $M$ be a $T-$manifold and $\tau\colon T \to T$ an automorphism of $T$. We will begin
our derivation of $(1.11)$ by describing how the equivariant cohomology ring, $H_T(M)$
of $M$ is related to the ``$\tau-$twisted" equivariant cohomology ring $H_T(M)^{\tau}$, i.e. the cohomology ring associated with the action, $(g,m)\to \tau(g)m$.

The effect of this twisting is easiest to describe in terms of the Cartan model, $(\Omega_T(M),d_T)$.
We recall that in this model cochains are $T$ invariant polynomial maps
\begin{equation}\label{eq:5.1}
p\colon \mathfrak{t}\to \Omega(M),
\end{equation}
and the coboundary operator is given by
\begin{equation}\label{eq:5.2}
dp(\xi)=d(p(\xi))+\iota(\xi_M)p(\xi).
\end{equation}
``Twisting by $\tau$" is then given by the pull-back operation
$$
\tau^*\colon \Omega_T \to \Omega_T,\;\;\tau^*p(\xi)=p(\tau(\xi))
$$
 which converts $d_T$ to the coboundary operator
 \begin{equation}\label{eq:5.3}
 \tau^*d_T(\tau^{-1})^*=d+\iota(\tau^{-1}(\xi))=(d_T)^{\tau^{-1}},
 \end{equation}
 the expression on the right being the coboundary operator associated with the
 $\tau^{-1}$-twisted action of $T$ on $M$.
 
 Suppose now that $M$ and $N$ are $T-$manifolds and $f\colon M \to N$ a diffeomorphism
 which intertwines the $T-$action on $M$ with the $\tau-$twisted action on $N$.
 Then the pull-back map $f^*\colon \Omega(N)\to \Omega(M)$ satisfies
 $$
 \iota(\xi_M)f^*=f^*\iota(\tau(\xi)_N).
 $$
 Hence if we extend $f^*$ to the $\Omega_T$'s by setting $(f^*p)(\xi)=f^*(p(\xi))$
 this extended map satisfies
 \begin{equation}\label{eq:5.4}
 d_Tf^*=f^*(d_T)^{\tau}.
 \end{equation}
 Thus by \eqref{eq:5.3} and \eqref{eq:5.4} $\tau^*f^*$ intertwines the $d_T$ operators on
 $\Omega_T(N)$ and $\Omega_T(M)$ and hence defines an isomorphism on cohomology
 \begin{equation}\label{eq:5.5}
 \tau^{\#}f^{\#}\colon H_T(N)\to H_T(M).
 \end{equation}
 Moreover, for any diffeomorphism $f\colon M \to N$ (not just the $f$ above), the pull-back operation $(f^*p)(\xi)=f^*p(\xi)$ intertwines the $\tau^*$ operations, i.e. 
 \begin{equation}\label{eq:5.6}
 \tau^*f^*=f^*\tau^*.
 \end{equation}
 Another property of $f^*$ which we will need below if the following.
 If $T_e$ is a subgroup of $T$ one has restriction maps 
 $$
 \Omega_T\to \Omega_{T_e},\;\;p\to p_{|_{\mathfrak{t}_e}}
 $$
 and these induce maps in cohomology.
 If $\tau_{|T_e}$ is the identity it is easily checked that the diagram 
 $$
 \begin{array}{ccc}\label{eq:5.7}
 H_T(N) &\xrightarrow{\tau^{\#}f^{\#}} &H_T(M) \\
 \downarrow & & \downarrow \\
 H_{T_e}(N) & \xrightarrow{f^{\#}}& H_{T_e}(M)\\
 \end{array}
 $$
 commutes.
 
To apply these observations to the fibers of $(1.1)$  we begin by recalling a 
few elementary facts about holonomy. By equipping $M$ with a $T-$invariant 
Riemannian metric we get for each $m \in M$ an orthonormal complement in $T_mM$ 
of the tangent
space at $m$ to the fiber of $\pi$, \emph{i.e.} an ``Ehresman connection.'' 
Thus, if $p$ and $q$ are points of $B$ and $\gamma$ is a curve joining $p$ to $q$ we get, 
by parallel transport, a diffeomorphism $f_{\gamma} \colon F_p \to F_q$, where 
$f_{\gamma}(m)$ is the terminal point of the horizontal curve in $M$ projecting onto 
$\gamma$ and having $m$ as its initial point. Moreover, if $\gamma$ and $\gamma'$ 
are homotopic curves joining $p$ to $q$, then the diffeomorphisms 
$f_{\gamma}$ and $f_{\gamma'}$ are isotopic, \emph{i.e.} the isotopy class of $f_{\gamma}$ 
depends only on the homotopy class of $\gamma$.

Suppose now that the base $B$ is GKM, $p$ and $q$ are adjacent vertices of $\Gamma_B$, 
$e$ is the edge joining them and $S$ the two-sphere corresponding to this edge.
We can then choose $\gamma$ to be a longitudinal line on $S$ joining the South pole $p$ 
to the North pole $q$; since this line is unique up to homotopy, we get an intrinsically 
defined isotopy class of diffeomorphisms of $F_p$ onto $F_q$. Moreover since the Ehresman 
connection on $M$ is $T-$invariant and $T_e$ fixes $\gamma$, the maps in this istopy 
class are $T_e-$invariant. We will decree that the fibration (1.1) is \emph{balanced} 
if there exists a diffeomorphism $f_e$ in this isotopy class and an automorphism 
$\tau_e$ of $T$ such that $f_e$ intertwines the $T-$action on $F_p$ with 
the $\tau_e-$twisted action of $T$ on $F_q$. 

It is clear that this $\tau_e$, if it exists, has to be unique and has to restrict 
to the identity on $T_e$. Moreover, given a path
 $\gamma\colon e_1,\ldots,e_k$, in $\pi_1(\Gamma,p)$ we have for each $i$ a ring
 isomorphism 
 \begin{equation}\label{eq:5.7new}
 \tau_{e_i}^{\#}f_{e_i}^{\#}\colon H_T(F_{p_{i+1}})\to H_T(F_{p_i}),
 \end{equation}
 $p_i$ and $p_{i+1}$ being the initial and terminal vertices of $e_i$, and by composing
 these maps we get a ring automorphism, $\tau_{e_k}^{\#}f_{e_k}^{\#}\circ\cdots \circ\tau_{e_1}^{\#}f_{e_1}^{\#}$, of $H_T(F_p)$. Moreover, by \eqref{eq:5.6} we can rewrite the factors
 in this product as $\tau_{\gamma}^{\#}f_{\gamma}^{\#}$ where $\tau_{\gamma}=\tau_{e_1}\circ \cdots \circ \tau_{e_k}$ and $f_{\gamma}=f_{e_1}\circ \cdots \circ f_{e_k}$.
 Thus the map, $\gamma\to \tau_{\gamma}^{\#}f_{\gamma}^{\#}$ gives one a \textit{holonomy action} of $\pi_1(\Gamma,p)$ on $H_T(F_p)$. Alternatively letting $W_p$ be the image in 
 $Aut(F_p)\times Aut(T)$ of this map we can view this as a holonomy action of $W_p$ on $H_T(F_p)$.
 
 Now let $c_p$ be an element of $H_T(F_p)$ and $\gamma_{p,q}\colon e_1,\ldots, e_l$ a path
 in $\Gamma$ joining $p$ to $q$. Then one can parallel transport $c_p$ along $\gamma$ by
 the series of maps $\tau_{e_i}^{\#}f_{e_i}^{\#}$ to get an element $c_q$ in $H_T(F_q)$
 and if $c_p$ is in $H_T(F_p)^{W_p}$, this parallel transport operation doesn't depend
 on the choice of $\gamma$. Moreover if $q_1$ and $q_2$ are adjacent vertices and $e$
 is the edge joining $q_1$ to $q_2$, $c_{q_1}=\tau_e^*f_e^*c_{q_2}$ and hence by
 the commutative diagram above the images of $c_{q_1}$ and $c_{q_2}$ in the quotient
 space 
 $$
 H_T(F_{q_1})/\langle \alpha_e \rangle = H_T(F_{q_2})/\langle \alpha_e \rangle
 $$
 are the same. In other words by \eqref{eq:new1.7} the assignment, $q\in Vert(F)\to c_q$
 defines a cohomology class in $H_T(M)$ and thus gives us a map
 \begin{equation}\label{eq:5.8}
 H_T(F_p)^{W_p}\to H_T(M)\;.
 \end{equation}
 By tensoring this map with the map 
 $$
 H_{\alpha}(\Gamma) \xrightarrow{\sim}H_T(B)\xrightarrow{\pi^{\#}}H_T(M)
 $$
 we get a morphism of rings \eqref{eq:newmain}:
 $$
 H_T(F_p)^{W_p}\otimes _{\SS(\mathfrak{t}^*)^{W_p}}H_{\alpha}(\Gamma)\to H_T(M)\; .
 $$
 To prove that this map is injective we will assume henceforth that not only is $M$ 
equivariantly formal as a $T$ space but the $F_p$'s are as well. Apropos of this assumption 
we note:

\begin{itemize}
\item[(i)] Since the fibration, $M\to B$, is balanced, it suffices to assume this just 
for the ``base" fiber, $F_{p_0}$, above a single $p_0\in B^T$.
\item[(ii)] One consequence of this assumption is that the cohomology groups $H_T^k(F)$
and $H^k(F)$ are non-zero only in even dimensions. Hence, since we are also assuming that
this is the case for $H_T(B)$, the Serre-Leray spectral sequence associated with the fibration
$(1.2)$ has to collapse at its $E_2$ stage and hence the right and left hand sides of $(1.4)$
are isomorphic as $\SS(\ft^*)$ modules.
\item[(iii)] For the homogeneous fibrations in Section 4 this assumption is equivalent to the 
assumption that the $F_p$'s are GKM. To see this we note that if $G/K$ is equivariantly formal 
then $(G/K)^T$ has to be non-empty by \cite{GS} theorem $11.4.5$ and hence for some $g\in G$, 
$g^{-1}Tg\subseteq K$. In other words $K$ is of maximal rank and hence by the theorem in 
\cite{GHZ} that we cited above $G/K$ is GKM.
\item[(iv)] If $M$ is a Hamiltonian $T-$space, then the fibers are Hamiltonian spaces as well, 
hence are equivariantly formal. (Notice that in particular if $M$ is a Hamiltonian GKM space,
then the fibers are Hamiltonian GKM spaces.)
\item[(v)] One consequence of the fact that $H_T(F)$ is equivariantly formal as a 
module over $\SS(\ft^*)$ is that $H_T(F_p)^W$ is equivariantly formal as a module 
over $\SS(\ft^*)^W$. It is interesting to note that this property of $F$ is a 
consequence of the equivariant formality of $M$, \emph{i.e.}  doesn't require the
assumption that $F$ be equivariantly formal. Namely to prove that $H_T(F)^W$ is 
equivariantly formal one has to show that there are no torsion elements in $H_T(F)^W$:
if $0\neq p \in \SS(\ft^*)^W$ and $c \in H_T(F_p)^W$, then $pc=0$ implies $c=0$. Suppose 
this were not the case. Then the cohomology class $\tilde{c} \in H_T(M)$ obtained from $c$ 
by parallel transport would satisfy $p\tilde{c}=0$, contradicting the assumption that $M$ is equivariantly formal.
\end{itemize}
  
  We next note that, since $B$ is simply connected, the diffeomorphisms $f_{\gamma}\colon
 F_p \to F_p$ are (non-equivariantly) isotopic to the identity, so they act trivially on $H(F_p)$,
 and since $F_p$ is by assumption equivariantly formal,
 \begin{equation}\label{eq:5.9}
 H_T(F_p)\simeq H(F_p)\otimes_{\mathbb{C}}\SS(\mathfrak{t}^*)
 \end{equation}
 as an $\SS({\mathfrak{t}^*})$ module. Hence if one chooses elements $c_1(p),\ldots,c_N(p)$
 of $H_T(F_p)$ whose projections, $c_i$, in $H(F_p)=H_T(F_p)/\SS(\mathfrak{t}^*)$ are a
 basis of $H(F_p)$, these will be a free set of generators of $H_T(F_p)$ as a module over
 $\SS(\mathfrak{t}^*)$. Moreover, we can average these generators by the action of $W_p$ and
 by the remark above these averaged generators will have the same projections onto $H(F_p)$.
 Hence we can assume, without loss of generality, that the $c_i(p)$'s themselves are in 
 $H_T(F_p)^{W_p}$ and by \eqref{eq:5.9} generate $H_T(F_p)^{W_p}$ as a module over
 $\SS(\mathfrak{t}^*)^{W_p}$.
 
 If we parallel transport these generators to the fiber over $q$, we will get a set of generators,
 $c_1(q),\ldots,c_N(q)$ of $H_T(F_q)$, and the maps $q\to c(q)$ define, by Chang-Skjelbred,
 cohomology classes $\tilde{c}_k$ in $H_T(M)$.
 
 Consider the map
 \begin{equation}\label{eq:5.10}
 \sum (c_i\otimes f_i,q)\to \sum f_i(q)c_i(q)
 \end{equation}
 of $H(F_p)\otimes_{\mathbb{C}}H_{\alpha}(\Gamma)$ into $H_T(M)$. Since the $c_i(q)$'s
 are, for every $q\in Vert(\Gamma)$, a free set of generators of $H_T(F_q)$ as a module over
 $\SS(\mathfrak{t}^*)$,
 this map is an injection and hence so is the equivariant version of this map: $(1.11)$.
 To see that injectivity implies surjectivity we note that, if we keep track of bi-degrees, the map
 \eqref{eq:5.10} maps the space
 \begin{equation}\label{eq:5.11}
\bigoplus_{j+k=i}H^j(F_p)\otimes H_{\alpha}^k(\Gamma) 
 \end{equation}
 into $H_T^i(M)$. However, by assumption, the Serre-Leray spectral sequence associated
 with the fibration $\pi$ collapses at its $E_2$ stage. The $E_2$ term of this sequence is
 \begin{equation}\label{eq:5.12}
 H(F_p)\otimes_{\mathbb{C}}H_T(B)
 \end{equation}
 and the $E_{\infty}$ term is $H_T(M)$, so by \eqref{eq:5.11} and \eqref{eq:5.12} the space
 \eqref{eq:5.11} has the same dimension as $H_T^i(M)$ and hence \eqref{eq:5.10}
 is a bijective map of \eqref{eq:5.11} onto $H_T^i(M)$.
 
\section{Examples}\label{sec:6}

The results of this paper are closely related to the combinatorial
results of our recent article \cite{GSZ}.  More explicitly in \cite{GSZ} we develop a
GKM theory for fibrations in which the objects involved: the base, the
fiber and the total space of the fibration, are GKM graphs. We then
formulate, in this context, a combinatorial notion of ``balanced,'' show
that one has an analogue of the isomorphism (1.11) and  use this fact to
define some new combinatorial invariants for the classical flag varieties
of type $A_n$, $B_n$, $C_n$, and $D_n$. In this section we will give a brief
account of how these invariants can be defined geometrically by means of
the techniques developed above.

\begin{exm}{\rm Let $G=SU(n+1)$, $K=T$, the Cartan subgroup of diagonal matrices in $SU(n+1)$, and 
$K_1=S(S^1\times U(n))$. Then $G/K_1 \simeq \CC P^{n}$, the complex projective space. Let $\mathcal{A}_n=G/K$ be the generic coadjoint orbit of type $A_n$; then 
$\mathcal{A}_n \simeq \mathcal{F}l(\CC^{n+1})$, the variety of complete flags in $\CC^{n+1}$. 
The fibration
$$\pi \colon \mathcal{A}_n\simeq G/K \to G/K_1 \simeq \CC P^{n}$$
sends a flag $V_{\bullet}$ to its one-dimensional subspace $V_1$. 
The fiber over a line $L$ in $\CC P^{n}$ 
is $\mathcal{F}l(\CC^{n+1}/L) \simeq \mathcal{F}l(\CC^{n}) \simeq SU(n)/T'$, where $T'$ is 
the Cartan subgroup of diagonal matrices in $SU(n)$. The fibers inherit a $T-$action 
from $\mathcal{F}l(\CC^{n+1})$, but are not $T-$equivariantly isomorphic. If $p$ and $q$ are fixed 
points for the $T-$action on $\CC P^{n}$, then the fibers $F_p$ and $F_q$ are $T_e$-equivariantly 
isomorphic, where $T_e$ is the subtorus fixing the $\CC P^1$ in $\CC P^{n}$ with poles $p$ and $q$. 
The Weyl group $W_p$ of the fiber at $p$ is isomorphic to $S_{n-1}$, the Weyl group of $SU(n)$ and
the holonomy action action of $W_p$ on the equivariant cohomology of the fiber is equivalent 
to the induced action of $S_{n-1}$ on the equivariant cohomology of the 
flag variety $\mathcal{F}l(\CC^n)$.

We can iterate this fibration and construct a tower of fiber bundles
%
$$
\begin{CD}
\text{pt}\!\! @>>>\!\! \mathcal{A}_1@>>>\!\!\mathcal{A}_2\!\! @>>>\!\! \ldots\!\! 
@>>>\!\! \mathcal{A}_{n-1} \!\!@>>> \!\!\mathcal{A}_n\\
                   @.      @VVV            @VVV                   @.              @VVV                       @VVV \\
                   @.      \CC P^1   @.       \CC P^2         @.         @.              \CC P^{n-1}      @.          \CC P^{n}
\end{CD}
$$

Using this tower we construct a basis of invariant classes on $\mathcal{A}_n$ by 
repeatedly applying the isomorphism \eqref{eq:newmain}. 
A typical stage in the process 
is the following. By \eqref{eq:newmain} we have
$$H_T(\mathcal{A}_k) = 
H_T(\mathcal{A}_{k-1})^{W_{k-1}} \otimes_{\SS(\ft^*)^{W_{k-1}}} H_{T}(\CC P^{k})$$

Suppose we have constructed a basis of invariant classes 
on $\mathcal{A}_{k-1}$; this is trivial for $\mathcal{A}_0=$pt. We use, 
as a basis for $H_T(\CC P^{k})$, classes represented by powers of the equivariant 
symplectic form $\omega - \tau\in \Omega_T^2(\CC P^{k})$. The pull-backs of 
these classes to $\mathcal{A}_k$ are invariant under the holonomy action, and 
the classes given by the isomorphism \eqref{eq:newmain} form a basis of the 
equivariant cohomology of $\mathcal{A}_k$. As shown in \cite{GSZ}, this basis consists of classes that 
are invariant under the corresponding holonomy action.
By iterating this process we obtain an $\SS(\ft^*)$-basis of $H_T(\mathcal{A}_n)$ 
consisting of $W-$invariant classes. The combinatorial version of this construction 
is given in \cite[Section 5.1]{GSZ}. 
}\end{exm}

\begin{exm}{\rm 
Let $G=SO(2n+1)$, $K=T$ a maximal torus in $G$, and $K_1=SO(2)\times SO(2n-1)$.
Then $G/K_1\simeq Gr_2^+(\RR^{2n+1})$, the Grassmannian of oriented two planes in $\RR^{2n+1}$. Let $\mathcal{B}_n=G/K$ be the generic coadjoint orbit of type $B_n$ and 
$$\pi \colon \mathcal{B}_n \simeq G/K \to G/K_1 \simeq Gr_2^+(\R^{2n+1})$$
the natural projection. Since the fibers are isomorphic to $\mathcal{B}_{n-1}$
(but not isomorphic as $T-$spaces since the $T-$action on the pre-image of the $T-$fixed points of $Gr_2^+(\R^{2n+1})$ changes), we can produce a tower of fiber bundles
$$
\begin{CD}
\text{pt}\!\! @>>>\!\! \CC P^1\!\!@>>>\!\!\mathcal{B}_2\!\! @>>>\!\! \ldots\!\! 
@>>>\!\! \mathcal{B}_{n-1}\!\!@>>> \!\!\mathcal{B}_n\\
                   @.      @VVV            @VVV                   @.              @VVV                       @VVV \\
                   @.      \CC P^1   @.       Gr_2^+(\R^5)         @.         @.              Gr_2^+(\R^{2n-1})     @.          Gr_2^+(\R^{2n+1})
\end{CD}
$$
Since the classes represented by powers of the equivariant symplectic form 
$\omega - \tau\in \Omega_T^2(Gr_2^+(\R^{2k-1}))$ form a $W-$invariant basis for 
$H_T(Gr_2^+(\R^{2k-1}))$, we can repeat the same argument of the previous
example and produce a basis for $H_T(\mathcal{B}_n)$ consisting of $W-$invariant classes.
}\end{exm}

\begin{exm}{\rm 
Let $G=Sp(n)$ be the symplectic group, $K=T$ a maximal torus in $G$, and
$K_1=S^1\times Sp(n-1)$; then $G/K_1\simeq \CC P^{2n-1}$. Let $\mathcal{C}_n=G/K$
be the generic coadjoint orbit of type $C_n$ and 
$$\pi \colon \mathcal{C}_n \simeq G/K \to G/K_1 \simeq \CC P^{2n-1}$$
the natural projection. Then, since the fibers are isomorphic to $\mathcal{C}_{n-1}$, 
we obtain the following tower of fiber bundles
$$
\begin{CD}
\text{pt}\!\! @>>>\!\! \CC P^1\!\!@>>>\!\!\mathcal{C}_2\!\! @>>>\!\! \ldots\!\! 
@>>>\!\! \mathcal{C}_{n-1}\!\!@>>> \!\!\mathcal{C}_n\\
                   @.      @VVV            @VVV                   @.              @VVV                       @VVV \\
                   @.      \CC P^1   @.       \CC P^3        @.         @.              \CC P^{2n-3}    @.          \CC P^{2n-1}\end{CD}
$$
By taking classes represented by powers of the equivariant symplectic form $\omega-\tau\in \Omega_T^2(\CC P^{2k-1})$ we obtain a $W-$invariant basis of $H_T(\CC P^{2k-1})$, and iterating
the same procedure as before, a $W-$invariant basis of $H_T(\mathcal{C}_n)$.
}\end{exm}

\begin{exm}{\rm 
Let $G=SO(2n)$, $K=T$ a maximal torus in $G$, and $K_1=SO(2)\times SO(2n-2)$.
Then $G/K_1\simeq Gr_2^+(\R^{2n})$, the Grassmannian of oriented two planes in $\R^{2n}$.
Let $\mathcal{D}_n=G/K$
be the generic coadjoint orbit of type $D_n$ and 
$$\pi \colon \mathcal{D}_n \simeq G/K \to G/K_1 \simeq Gr_2^+(\R^{2n}) $$
the natural projection. Then, since the fibers are isomorphic to $\mathcal{D}_{n-1}$, 
we obtain the following tower of fiber bundles
$$
\begin{CD}
\text{pt}\!\! @>>>\!\! \CC P^1\times \CC P^1\!\!@>>>\!\!\mathcal{D}_3\!\! @>>>\!\! \ldots\!\! 
@>>>\!\! \mathcal{D}_{n-1}\!\!@>>> \!\!\mathcal{D}_n\\
                   @.      @VVV            @VVV                   @.              @VVV                       @VVV \\
                   @.      \CC P^1\times \CC P^1   @.       Gr_2^+(\R^6)        @.         @.              Gr_2^+(\R^{2n-2})    @.          Gr_2^+(\R^{2n})
                   \end{CD}
$$
}\end{exm}

In \cite{GSZ} we also show how these iterated invariant classes relate to a better known family 
of classes generating the equivariant cohomology of flag varieties, namely the equivariant Schubert 
classes.

\end{document}